\documentclass[]{interact}

\usepackage{epstopdf}
\usepackage[caption=false]{subfig}

\usepackage[numbers,sort&compress]{natbib}
\bibpunct[, ]{[}{]}{,}{n}{,}{,}
\renewcommand\bibfont{\fontsize{10}{12}\selectfont}
\makeatletter
\def\NAT@def@citea{\def\@citea{\NAT@separator}}
\makeatother

\theoremstyle{plain}
\newtheorem{theorem}{Theorem}[section]
\newtheorem{lem}[theorem]{Lemma}

\newtheorem{prop}[theorem]{Proposition}
\newtheorem{ass}[theorem]{Assumption}
\theoremstyle{definition}
\newtheorem{deff}[theorem]{Definition}

\theoremstyle{remark}
\newtheorem{rem}{Remark}

\usepackage{mathrsfs}
\usepackage{dsfont}

\usepackage{amsthm}

\newcommand{\eps}{\varepsilon}
\newcommand{\ups}{\upsilon}
\newcommand{\hh}{\mathcal{H}}
\newcommand{\as}{\mathcal{A}}
\newcommand{\ls}{\mathcal{L}}
\newcommand{\vs}{\mathcal{V}}

\newcommand{\os}{\mathcal{O}}

\newcommand{\hs}{\mathcal{H}}
\newcommand{\R}{\mathbb{R}}

\newcommand{\ppp}{\mathbf{P}}

\newcommand{\eee}{\mathbf{E}}
\newcommand{\qf}{\mathfrak{q}}
\newcommand{\cf}{\mathfrak{c}}
\newcommand{\ef}{\mathfrak{er}}
\newcommand{\pf}{\mathfrak{p}}

\newcommand{\tu}{\tilde{u}}

\newcommand{\dom}{\operatorname{dom}}
\newcommand{\zbar}{\bar{z}}

\newcommand{\pk}{\mathfrak{p}}

\newcommand{\yk}{\mathfrak{y}}

\newcommand{\K}{\operatorname{K}}
\newcommand{\trans}{\mathcal{T}}

\newcommand{\Qs}{\mathcal{Q}}

\newcommand{\tr}{\operatorname{tr}}

\newcommand{\LL}{\mathcal{L}}
\newcommand{\LLk}{\mathfrak{L}}

\usepackage{color}

\newcommand{\ra}{\rightarrow}

\newcommand{\hk}{\mathfrak{h}}

\newcommand{\lle}{\LLk^{\qf,\eps}}

\newcommand{\Ws}{\mathcal{W}}

\newcommand{\G}{\mathfrak{G}}

\newcommand{\zero}{\mathbf{0}}

\newcommand{\hha}{\hh_{\alpha,s}}

\newcommand{\ym}[1]{{\color{black} #1}}

\begin{document}


\title{Almost Sure Asymptotic Stability of Parabolic SPDEs with Small Multiplicative Noise} 

\author{
\name{Yiming Meng\textsuperscript{a}, N.Sri.Namachchivaya\textsuperscript{a}\thanks{CONTACT Yiming Meng. Email: yiming.meng@uwaterloo.ca, ymmeng@illinois.edu.}, and Nicolas Perkowski\textsuperscript{b}}
\affil{\textsuperscript{a}Department of Applied Mathematics,  University of Waterloo,  ON, Canada; \\\textsuperscript{b}Institut f{\"u}r Mathematik, Freie Universit{\"a}t Berlin, Berlin, DE.}
}

\maketitle

\begin{abstract}
A better understanding of the instability margin  will eventually optimize the 
operational range for safety-critical industries. In this paper, we investigate the almost-sure exponential asymptotic stability of the trivial solution of a  parabolic semilinear  stochastic partial differential equation (SPDE) driven by multiplicative noise near the deterministic Hopf bifurcation point. We show the existence and uniqueness of the invariant measure under appropriate assumptions, and approximate the exponential growth rate via asymptotic expansion, given that the strength of the noise is small.  This approximate quantity can readily serve
as a robust indicator of the change of almost-sure stability. 
\end{abstract}


\section{Introduction}
The almost-sure asymptotic stability/instability of the trivial solution is captured by the sign of the top Lyapunov exponent $\lambda$ of the linearized system. In the contexts of stochastic bifurcation, under proper conditions, for finite-dimensional SDEs with coefficients dependent on some parameter $\gamma$, if $\gamma$ varies in a way that $\lambda(\gamma)$ changes sign from negative to positive, the trivial solution loses its almost-sure asymptotic stability and a nontrivial invariant measure is formed. This stochastic Hopf bifurcation is investigated by \cite{baxendale1994stochastic}. 

In terms of approximations, for systems driven by multiplicative noise, multiscale analysis and stochastic averaging/homogenization techniques have been applied to dimensional reduction problems of noisy nonlinear
systems with rapidly oscillating and decaying components. For finite-dimensional SDEs with small multiplicative noise, when  many modes are ‘heavily damped’, reduced-order
models were obtained using a martingale problem
approach  in \cite{namachchivaya2004averaging}. The result verifies that  a lower dimensional Markov process characterizes the limiting behavior in the weak topology as the noise
becomes smaller.  In application, the authors of \cite{singh2016stochastic} derived a low-dimensional
approximation of an
11-dimensional nonlinear stochastic aeroelastic problem near 
a deterministic Hopf bifurcation, with one critical mode and
several stable modes. The reduced model performs well in terms of simulating the distribution, density, as well as the top Lyapunov exponent of the full system near the deterministic Hopf bifurcation.

In contrast to the finite-dimensional cases, it is difficult to quantitatively describe the random invariant manifolds and stochastic bifurcations for SPDEs driven by multiplicative noise  \cite{guo2016approximation,sun2010impact,li2018stochastic,chekroun2014stochastic}. However, it is worth mentioning that there are several results regarding the approximation of the transient dynamics within the critical manifold on a sufficiently long time scale. 
For example, the work presented in~\cite{blomker2003amplitude,blomker2020impact,blomker2007multiscale}  describes the phenomenological bifurcation (defined in \cite{arnold1996toward}) using amplitude equations and multi-scale analysis techniques. A more rigorous analysis of the approximation of the invariant measure was developed in \cite{blomker2004multiscale}. In terms of stability for infinite-dimensional cases, the almost-sure stability of scalar stochastic delay differential equations have been studied in \cite{sri2012almost}. 

The recent work \cite{blomker2020impact} studies the impact of multiplicative noise in semilinear {SPDE}s near the bifurcation using amplitude equations. However, the hypothesis only guarantees a local existence of a mild solution to the {SPDE}s up to a random explosion time. Under a stronger dissipative assumption in the critical modes, the derived amplitude equation can be used to approximate the solution up to a fixed deterministic time and, hence, predict the stability of the original {SPDE}. The approximation error  converges in probability.

Unlike \cite{namachchivaya2004averaging,singh2016stochastic,blomker2020impact}, we  do not attempt to use averaging/homogenization and dimension reduction to study the structural change of the invariant measures.  In this paper, as a necessary step to study the dynamical bifurcation (defined in \cite{arnold1996toward}) in infinite dimensions, we aim to investigate the almost-sure asymptotic stability of the trivial solutions of SPDEs driven by multiplicative noise. In particular, we study the effect of multiplicative noise near the Hopf bifurcation point (quasi-static) $\gamma_c$  of the unperturbed system. While it remains challenging to obtain the exact expressions for top Lyapunov exponents,  
we  extend the multiscale analysis in \cite{sri2012almost} and find the asymptotic approximation of them. To be more precise, the multiscale analysis is conducted in the neighbourhood of the deterministic bifurcation point $\gamma_c$, i.e., $\gamma=\gamma_c+\eps^2\qf$ with some $\qf\in\R$ and $\eps\ll 1$.
We directly consider 
linear (linearized) SPDEs of the following general form

\begin{equation}\label{E: abstract}
 du(t)=\as(\gamma)u(t)dt+\eps G(u(t))dW(t),
\end{equation}
where $u(t)$ takes value in an infinite-dimensional separable Hilbert space $\hh=L^2(E)$ for some bounded $E\subseteq\R^n$. The self-adjoint
unbounded linear $\as(\gamma)$ that also depends on a parameter $\gamma\in\R$  generates an analytic compact $C_0$ semigroup on $\hh$. The operator $G(u)$ is Hilbert-Schmidt  with $G(\zero)=\zero$. The noise $W$ is a cylindrical Wiener process with intensity of $\eps$.
 
It is worth noting that even a small intensity of stochastic perturbations from the stable modes will shift  both  dynamical and
phenomenological bifurcation points. The accuracy of the approximation matters in safety-critical industries. Upon improving the approximation precision, as we will show in this paper, even though the stable modes quickly decay and can be ignorable, we do not cancel the coupling effects in the multiplicative noise  so as to keep  interactions between the critical and stable modes as in the original dynamics. A closely related work \cite[Chapter 6-8]{meng2022bifurcation} has shown that 
such a coupling has a slight impact on the D-bifurcation point
as well as on the P-bifurcation point via second-order corrections.   
We also pursue in our future work to use this  approximation scheme of the top Lyapunov exponents  along with the nonlinear semilinear SPDEs  to analyze the structural changes in random attractors as the trivial solution loses its stability.

The rest of this paper is organized as follows. 
 In Section \ref{sec: intro2}, we introduce the notations and formulate necessary assumptions for the  analysis. In Section \ref{sec: stability}, we investigate the almost-sure asymptotic stability of the trivial solution. In particular, we derive the Furstenberg–Khasminskii formula for the top Lyapunov exponent and show the existence and uniqueness of the invariant measure under appropriate conditions.  
We approximate the exponential growth rate via asymptotic expansion in Section \ref{sec: lya_approx}.
The conclusions follow in
Section \ref{sec: conclusions}.

\section{Notations and Main Assumptions}\label{sec: intro2}

Given the separable Hilbert space $\hh$, we denote by $\langle\cdot,\;\cdot\rangle$ the inner product in $\hh$ and by $\|\cdot\|$ the norm. We also identify $\hh$ with its dual through
the Riesz isomorphism. Due to the compactness of the associated semigroup, $\as(\gamma)$ has a purely discrete point
spectrum. We make the following assumptions on $\as(\gamma)$  about the point spectrum.
\begin{ass}
\label{ass: A1-2}
We assume that, for all $\gamma$,
\begin{enumerate}
    \item[(1)] the point spectrum $\{\rho_{ k}(\gamma)\}_{k\in \mathbb{Z}_0}$ of $\as(\gamma)$ is complex, where $\rho_{ k}(\gamma):=a_k(\gamma)+ib_k(\gamma)\in\mathbb{C}$ and $\rho_{ -k}(\gamma)=\overline{\rho}_k(\gamma)$ for all $k\in\mathbb{Z}_0$;
   
    \item[(2)] $\rho_k(\gamma)$ is analytic
    in $\gamma$ for all $\gamma\in\R$ and all $k\in\mathbb{Z}_0$,
  and $a_1(\gamma)> a_2(\gamma)\geq \dots\geq a_k(\gamma)\geq \dots$;
   \item[(3)] the corresponding eigenvectors $\{e_k\}_{k\in \mathbb{Z}_0}$ form a complete orthonormal basis of $\mathcal{H}$ such that $\as(\gamma) e_k=\rho_k(\gamma)e_k$,  $\langle e_{-k},e_k\rangle=1$ for all $k\in\mathbb{Z}_0$, and $\langle e_{i},e_j\rangle=0$ for all $i+j\neq 0$. 
\end{enumerate}
\end{ass}
\begin{rem}
Note that by Assumption \ref{ass: A1-2}, $\as u=\sum_{k\in\mathbb{Z}_0}\rho_k\langle e_{-k}, u\rangle e_k$ for all $u\in\hh$, and for all $u,v\in\hh$ we have
\begin{equation*}
    \begin{split}
        \langle \as u,v\rangle &=\sum_{k\in\mathbb{Z}_0}\rho_k\langle e_{-k}, u\rangle \langle e_k, v\rangle=\sum_{k\in\mathbb{Z}_0}\rho_{-k}\langle e_k, u\rangle \langle e_{-k}, v\rangle\\
        &=\sum_{k\in\mathbb{Z}_0}\rho_{k}\langle e_k, u\rangle \langle e_{-k}, v\rangle=\langle u,\as v\rangle,
    \end{split}
\end{equation*}
which indicates the self-adjoint property of $\as$.
The second to the last identity is in that, for each $k\neq 0$, 
$$\rho_{-k}\langle e_k, u\rangle \langle e_{-k}, v\rangle+\rho_{k}\langle e_{-k}, u\rangle \langle e_{k}, v\rangle=\rho_{k}\langle e_k, u\rangle \langle e_{-k}, v\rangle+\rho_{-k}\langle e_{-k}, u\rangle \langle e_{k}, v\rangle\in\R. $$
\end{rem}

With the appearance of noise that is white in time, either white or
colored in space, the regularity, especially the spatial differentiability, of the solution varies. We introduce the concept of fractional power space, which  renders a more flexible scale of  regularity.  

\begin{deff}[Fractional Power Space]\label{def: fractional-2}
For $\alpha\in \R$, given the linear operator $\as(\gamma)$, define the interpolation fractional power (Hilbert) space \cite{pazy2012semigroups} $\hh_{\alpha}:=\dom(\as^{\alpha}(\gamma))$ endowed with inner product $\langle u,v\rangle_{\alpha}= \langle \as^{\alpha}u,\as^{\alpha}v\rangle$ and corresponding induced norm $\|\cdot\|_{\alpha}:=\|\as^{\alpha}\cdot\|$.  Further more, we denote the dual space of $\hh_{\alpha}$ by $\hh_{-\alpha}$ w.r.t. the inner product in $\hh$.
\end{deff}

\begin{rem}
    Note that, for all $\alpha \geq \beta$, we have the inclusion $\hh_\alpha\subset \hh_\beta$. In addition, for any $\alpha>0$, we have $\hh_\alpha\subset \hh\subset \hh_{-\alpha}$.  We also kindly refer readers to \cite{hairer2009introduction, blomker2007multiscale, meng2023hopf} for more properties of the fractional power space.  
\end{rem}

Since $\gamma_c$ is the deterministic Hopf bifurcation point, we have $a_{\pm 1}(\gamma_c)=0$ as well as $a_{\pm 1}'(\gamma_c)\neq 0 $, $b_{\pm 1}(\gamma_c)\neq 0$, whilst the rest of the spectrum stays in the left half-plane. We introduce the shorthand notation $\hk:=e_1$ and $\bar{\hk}:=e_{-1}$ to denote the critical eigenvectors. We denote by $\hk^*$ and $\bar{\hk}^*$ the associated adjoint eigenvectors , which satisfy
\begin{equation*}
    \langle \hk^*,\hk\rangle =1,\;\langle \hk^*,\bar{\hk}\rangle =0,\;\langle \bar{\hk}^*,\bar{\hk}\rangle =1, \;\langle \bar{\hk}^*,\hk\rangle =0.
\end{equation*}
Due to the existence of spectral gap, we also introduce the projections and basic properties of  $\as(\gamma)$.
\begin{deff}\label{def: projection_multiplicative-2}
The 
critical projection operator is defined as
\begin{equation}
    P_c(\cdot):=\langle \hk^*,\cdot\;\rangle\hk+\langle\bar{\hk}^*,\cdot\;\rangle \bar{\hk}.
\end{equation}
The stable projection operator is $P_s=I-P_c$. For simplicity, we introduce shorthand notation $\as_c(\gamma):=P_c\as(\gamma)$. We define $\as_s(\gamma)$, $G_c$, $G_s$, $\hh_c$, $\hh_s$, $\hh_{\alpha, c}$ and $\hh_{\alpha, s}$ in a similar way. 

\end{deff}

\begin{deff}[Other notations for $\as(\gamma)$]\label{def: other_notation-2} To this end, we use
\begin{itemize}
\item[(1)]$\mathbb{Z}_c:=\{\pm 1\}$, $\mathbb{Z}_s:=\mathbb{Z}_0\setminus\{\pm 1\}$.
\item[(2)] $\as_c^\cf:=\as_c(\gamma_c)$, $\as_c^\qf:=\qf\as_c'(\gamma_c)$; the associated eigenvalues of $\as_c^\cf$ and $\as_c^\qf$ are respectively denoted by $$\rho_{c}^\cf=ib_{c}^\cf:=ib_{1}(\gamma_c),\;\;\bar{\rho}_{c}^\cf=-ib_{c}^\cf:=-ib_{1}(\gamma_c),$$  
$$\rho_{c}^\qf=a_{c}^\qf+ib_{c}^\qf:=\qf(a_{1}'(\gamma_c)+ib_{1}'(\gamma_c))$$
and
$$\bar{\rho}_{c}^\qf=a_{c}^\qf-ib_{c}^\qf:=\qf(a_{1}'(\gamma_c)-ib_{1}'(\gamma_c)).$$
\item[(3)] $\as_c^\ef:=\eps^{-2}[\as_c(\gamma_c+\eps^2\qf)-\as_c^\cf-\eps^2\as_c^\qf]$, the associated eigenvalues of $\as_c^\ef$  are  denoted as
$\rho_{c}^\ef$ and $\bar{\rho}_{c}^\ef $.
    \item[(4)] $\as_s:=\as_s(\gamma)$ for all $\gamma$ and $\as^\cf:=\as_s+\as_c^\cf$.
\end{itemize}

\end{deff}

\begin{rem}
We introduce the second-order expansion $\as_c^\ef$ of $\as_c(\gamma)$ around $\gamma_c$ to better understand the effect in the multiscale expansion when the parameter of the linear operator is close to the deterministic Hopf bifurcation point. 
\end{rem}

\begin{ass}\label{ass: A2-2}
We assume that, for each $\gamma\in\R$, 
$\as(\gamma)$ generates an analytic compact $C_0$ semigroup $\{e^{t\as(\gamma)}\}_{t\geq 0}$ on $\hh$, which also commute with the critical projection operator $P_c$. We further assume that 

\begin{itemize}
    \item[(1)] There exists some $M_s>0$ and $c_s>0$ such that 
for all $u\in\hh_s$,
$$\|e^{t\as(\gamma)}u\|\leq Me^{-c_st}\|u\|,\;\;\forall t\geq 0.$$
\item [(2)] There exists  $M>0$ and $c\geq 0$ such that  for all $u\in\hh$ and each $\gamma$, 
$$\|e^{t\as(\gamma)}u\|\leq Me^{ct}\|u\|,\;\;\forall t\geq 0. $$
\end{itemize}
\end{ass}

\begin{prop}\label{prop: er_bound-2}
For each $u\in\hh$ and for all $\eps\in(0,1)$, there exist some $C_\qf>0$ and $C_\ef>0$ such that
$$\langle \as_c^\qf u,u\rangle\leq C_\qf\|P_c u\|^2 \;\;\text{and}\;\;\langle \as_c^\ef u,u\rangle\leq \eps^2C_\ef\|P_c u\|^2.$$
\end{prop}
\begin{proof}
It is clear from the definition of
projection that 
\begin{equation}
\begin{split}
    \langle \as_c^\qf u,u\rangle& =\langle \as_c^\qf (P_cu+P_su),P_cu+P_su\rangle=\langle \as_c^\qf (P_cu),P_cu\rangle\\
    &=\rho_c^\qf\langle \bar{\ups}, P_cu\rangle \langle \ups, P_cu\rangle+\bar{\rho}_c^\qf\langle \ups, P_cu\rangle \langle \bar{\ups}, P_cu\rangle\\
    &\leq 2|a_c^\qf|^2\cdot \|P_cu\|^2.
\end{split}
\end{equation}
The bound for  $\langle \as_c^\ef u,u\rangle$ can be obtained in a similar way with by the Cauchy-Taylor expansion based on the additional analyticity of $\rho_{\pm 1}(\gamma)$ as in (2) of Assumption \ref{ass: A1-2}.
\end{proof}

Now, we move on to the assumptions on the last term in \eqref{E: abstract}. Let $\ls_2(E,K)$ denote the set of all Hilbert-Schmidt operators from $E$ to $K$ for any separable Hilbert spaces $E$ and $K$. We write $\ls_2(E)$ instead of $\ls_2(E,E)$ for short. We denote by $\|\cdot\|_{\ls_2(E,K)}$  the norm for Hilbert-Schmidt operators. If the spaces $E$ and $K$ are not emphasized, we also use the shorthand notation $\|\cdot\|_{\ls_2}$.

\begin{ass}
Let $\vs$ be a separable Hilbert space with orthonormal basis $\{\mathfrak{z}_k\}_{k\in\mathbb{Z}_0}$. We assume $W$ is a $\vs$-valued cylindrical Wiener process. 
\end{ass}

\begin{ass}\label{ass: G-2}
Assume that, for each $\alpha\in(0,1]$ and for all $u\in\mathcal{H}_\alpha$, $G(u)$ is a  Hilbert-Schmidt operator from $\vs$ to $\hh_\alpha$, i.e.,
$G: \mathcal{H}_\alpha\rightarrow \ls_2(\vs,\hh_\alpha)$. 
We further require that $G(u)$ is Fr\'{e}chet-differentiable satisfying
\begin{enumerate}
    \item[(1)] there exists an $\ell_1>0$ s.t. $\|G(u)\|_{\ls_2(\vs,\hh_\alpha)}\leq \ell_1\|u\|_{\alpha}$ for all $ u\in\mathcal{H}_\alpha$;
    \item[(2)] there exists an $\ell_2>0$ s.t. $\|G'(u)\cdot v\|_{\ls_2(\vs,\hh_\alpha)}\leq \ell_2\|v\|_{\alpha}$  for all $ v\in\mathcal{H}_\alpha$;
    \item[(3)] $G''(u)=0$ for all $u\in \mathcal{H}$. 
\end{enumerate}
\end{ass}

\begin{rem}
Given an orthonormal basis $\{\mathfrak{z}_k\}_{k\in\mathbb{Z}_0}$ of  $\mathcal{V}$, we can write $$G(u)=\sum_{j\in\mathbb{Z}_0}\sum_{k\in\mathbb{Z}_0} g_{jk}(u)e_j\otimes \mathfrak{z}_k,$$ where $g_{ik}(u)\in\R$ is the eigenvalue of the operator $G(u)$ for all $u\in\hh_\alpha$ and for all $j,k\in\mathbb{Z}_0$.
\end{rem}

\section{Stability Analysis of the Trivial Solution} \label{sec: stability}

In this section, we investigate the almost-sure asymptotic stability of the trivial solution $\mathbf{0}$, or equivalently, the trivial invariant measure $\delta_{\mathbf{0}}$, using multiscale techniques. 
We recall that $\gamma=\gamma_c+\eps^2\qf$ with some unfolding parameter $\qf\in\R$, and introduce \ym{the notion of solution under re-scaled time} as follows, i.e. we introduce \ym{$$z(t)=\langle\hk^*,u(\eps^{-2}t)\rangle \;\;\text{and}\;\;\zbar(t)=\langle \bar{\hk}^*,u(\eps^{-2}t)\rangle$$ as the complex amplitudes of the critical mode 
and $y(t)= P_su(\eps^{-2}t)$}. Then the solution $u$ of \eqref{E: abstract} can be written as \ym{$$u(t)= z(\eps^2t)\hk+ \zbar(\eps^2t)\bar{\hk}+ y(\eps^2t).$$} We denote the real part and imaginary part of $z$ as $z_1=\operatorname{Re}(z)$ and $z_2=\operatorname{Im}(z)$, respectively.

Note that when the system is close to the critical point, 
due to the existence of the spectral gap, we decompose \eqref{E: abstract} into the re-scaled critical and fast-varying modes as follows:
\begin{equation*}
\begin{split}
& dz= \eps^{-2}\rho_c^\cf zdt + \rho_c^\qf  zdt+  \langle\hk^*,G(z\hk+\zbar\bar{\hk}+y)dW_t\rangle+ \rho_c^\ef  zdt,\\
& dy= \eps^{-2}\as_s ydt
+ G_s(z\hk+\zbar\bar{\hk}+y)dW_t,
\end{split}
\end{equation*}
or equivalently, 
\begin{subequations}\label{E:convertion_stall_noise2-2}
\begin{align}
& d\begin{bmatrix}
    z_1\\z_2\end{bmatrix}=\ym{\begin{bmatrix}
    0& -\varepsilon^{-2}b_c^\cf\\
    \varepsilon^{-2}b_c^\cf & 0
    \end{bmatrix}\begin{bmatrix}
    z_1\\z_2\end{bmatrix}+\begin{bmatrix}
    a_c^\qf& -b_c^\qf\\
    b_c^\qf & a_c^\qf
    \end{bmatrix}\begin{bmatrix}
    z_1\\z_2\end{bmatrix}}+\begin{bmatrix}
    G_c^R(z\hk+\zbar\bar{\hk}+y)\\G_c^I(z\hk+\zbar\bar{\hk}+y) \end{bmatrix}dW_t+\os(\eps^2),\label{E:convertion_stall_noise2a-2}\\
& dy= -\varepsilon^{-2}\as_s  ydt+  G_s(z\hk+\zbar\bar{\hk}+y)dW_t, \label{E:convertion_stall_noise2b-2}
\end{align}
\end{subequations}
where  $W(t)\sim\eps^{-1} W(\eps^{-2}t)$\footnote{We abuse the notation   to avoid redundancy.} and 
\ym{\begin{subequations}\label{E: G_c-2}
\begin{align}
& \hat{G}_1(u)w:=\langle \hk^*,G(u)w\rangle,\;\;\forall u\in \mathcal{H}_\alpha, \;\forall w\in \mathcal{V};\\
& G_c^R(u)=\frac{\hat{G}_1(u)+\overline{\hat{G}_{1}(u)}}{2},\;G_c^I(u)=\frac{\hat{G}_1(u)-\overline{\hat{G}_{1}(u)}}{2}.
\end{align}
\end{subequations}}
We impose the initial condition to \eqref{E:convertion_stall_noise2-2} as $z(0)=\langle\hk^*,u_0\rangle$ and $y(0)=P_su_0$. 

The truncation error term $\os(\eps^2) $ in \eqref{E:convertion_stall_noise2-2} comes from the transformation of $\rho_c^\ef z dt$, whose property has been verified in Proposition \ref{prop: er_bound-2}. \ym{In addition, given any $x(0)\in\hh_\alpha$  of order $ \os(1)$, one can easily verify by the boundedness of the operator $\as_c^\ef$ as well as the property of $\eee\sup_{0\leq t\leq T}\| x(t)\|_\alpha^p$ that, for any fixed time $T>0$ and fixed $p>0$, there exist constants  $C>0$ such that
\begin{equation}
    \eee\sup_{0\leq t\leq T}\left\|\int_0^t\as_c^\ef x(\sigma)d\sigma\right\|_\alpha^p\leq C\eps^{2p}.
\end{equation}}

We use the notation $\os(\eps^2)$ for short to indicate its order of error after integration. 
Due to the continuity of the term $\os(\eps^2) $ in $\eps$ and the insignificant effect, to this end, we work on the following equation to derive the first order approximation of the top Lyapunov exponent. 
\begin{subequations}\label{E: convertion_stall_noise2_new-2}
\begin{align}
& d\begin{bmatrix}
    z_1\\z_2\end{bmatrix}=\ym{\begin{bmatrix}
    0& -\varepsilon^{-2}b_c^\cf\\
    \varepsilon^{-2}b_c^\cf & 0
    \end{bmatrix}\begin{bmatrix}
    z_1\\z_2\end{bmatrix}+\begin{bmatrix}
    a_c^\qf& -b_c^\qf\\
    b_c^\qf & a_c^\qf
    \end{bmatrix}\begin{bmatrix}
    z_1\\z_2\end{bmatrix}}+\begin{bmatrix}
    G_c^R(z\hk+\zbar\bar{\hk}+y)\\G_c^I(z\hk+\zbar\bar{\hk}+y) \end{bmatrix}dW_t,\label{E: convertion_stall_noise2a_new-2}\\
& dy= -\varepsilon^{-2}\as_s  ydt+  G_s(z\hk+\zbar\bar{\hk}+y)dW_t. \label{E: convertion_stall_noise2b_new-2}
\end{align}
\end{subequations}

\subsection{The Furstenberg–Khasminskii Formula for the Top
Lyapunov Exponent}

Let
$$\mathfrak{p}(t)=\frac{1}{2}\ln(z_1^2(t)+z_2^2(t)),\;\;z_1(t)=e^{\pf(t)}\cos(\phi(t)),\;\;z_2(t)=e^{\pf(t)}\sin(\phi(t))$$ 
\text{and} $\eta_t=e^{-\pf(t)}y_t,$
where $\phi\in [0, 2\pi]$ is the phase angle in the unit sphere \cite{khas1967necessary} 
satisfying
\begin{equation}
    z_1=|z|\cos(\phi),\;\; z_2=|z|\sin(\phi) .
\end{equation}
Therefore, by It\^{o}'s formula,
\begin{subequations}\label{E: dyn_trans}
\begin{align}
 & d\pf      =a_c^\qf dt+\Xi(\phi,\eta)dt+[G_c^R(\phi,\eta)\cos\phi+G_c^I(\phi,\eta)\sin\phi]dW_t,\label{E: dyn_trans_a}\\
 & d\phi =(\varepsilon^{-2}b_c^\cf+b_c^\qf)dt+\Gamma (\phi,\eta) dt-G_c^\phi(\phi,\eta)dW_t\label{E: stable_phi},\\
        & d\eta = \varepsilon^{-2}\as_s\eta dt+ G_s(\phi,\eta)dW_t\label{E: stable_eta},
\end{align}
\end{subequations}
where\footnote{We abuse the notation $G$   and recast the arguments as $\phi$ and $\eta$.}
$$G_c^R(\phi,\eta):=G_c^R(\cos(\phi)\hk+\sin(\phi)\bar{\hk}+\eta),\;\;G_c^I(\phi,\eta):=G_c^I(\cos(\phi)\hk+\sin(\phi)\bar{\hk}+\eta),$$
$$G_s(\phi,\eta):=G_s(\cos(\phi)\hk+\sin(\phi)\bar{\hk}+\eta),\;\;G_c^\phi(\phi,\eta):=G_c^R(\phi,\eta)\sin\phi-G_c^I(\phi,\eta)\cos\phi,$$
and 
$$\Xi:=-\frac{\cos(2\phi)}{2}\tr[G_c^R(G_c^R)^*-G_c^I(G_c^I)^*](\phi,\eta)-\frac{\sin(2\phi)}{2}\tr[G_c^R(G_c^I)^*+G_c^I(G_c^R)^*](\phi,\eta),$$
$$\Gamma:=\frac{\sin(2\phi)}{2}\tr[G_c^R(G_c^R)^*-G_c^I(G_c^I)^*](\phi,\eta)-\frac{\cos(2\phi)}{2}\tr[G_c^R(G_c^I)^*+G_c^I(G_c^R)^*](\phi,\eta). $$
We also name $G_c^\pk(\phi,\eta):=G_c^R(\phi,\eta)\cos\phi+G_c^I(\phi,\eta)\sin\phi$ for future references. 

The initial condition is such that $\pf_0=\pf(0)=\ln|z(0)|$, $\phi_0=\phi(0)=\arctan\left(\frac{z_2(0)}{z_1(0)}\right)$ and $\eta_0=\eta(0)=e^{-\pf(0)}y(0)$.

We also denote the drift term of~\eqref{E: dyn_trans_a} as
\begin{equation}
    \mathcal{Q}^\qf(\phi,\eta):=a_c^\qf+\Xi(\phi,\eta).
\end{equation}
\begin{rem}\label{rem: rem}
Let $a:=\operatorname{tr}[G_c^R]$ and  $b:=\operatorname{tr}[G_c^I]$ for some fixed $(\phi,\eta)$, then $\Xi+i\Gamma=-e^{-2i\phi}(a+bi)^2$.
\end{rem}

Noticing that $\pf(t)$ only depends on $\phi(t)$ and $\eta(t)$, if there exists a unique invariant measure $\mu^\eps$ for the product process $(\phi(t),\eta(t))\in [0, 2\pi]\times\mathcal{H}_{s}$, the top Lyapunov exponent of $\eps^{-1}u$ can be determined by the Furstenberg–Khasminskii formula:
\begin{equation}\label{E: formula}
\begin{split}
       \lambda^{\qf,\eps}&=\lim_{t\rightarrow\infty}\frac{1}{t}\ln |z(t)|\\
        & =\int_{[0, 2\pi]\times\hh_s}\Qs^\qf(\phi,\eta) \mu^{\qf,\eps}(d\phi,d\eta)
        =:\langle \Qs^\qf,\mu^{\qf,\eps}\rangle.
\end{split}
\end{equation}
\ym{To this end, we show the existence and uniqueness of the invariant measure under appropriate conditions. }

\begin{rem}\label{rem: notaion_simplify}
Since we do not emphasize the variation $\qf$ as what we do in the study of bifurcation theory, to simplify the notation, we  use $\lambda^\eps$, $\Qs$ and $\mu^\eps$ instead for the derivations and proofs. 
\end{rem}

\subsection{Existence of Invariant Measure}
Note that \eqref{E: stable_phi} and \eqref{E: stable_eta} are coupled via the multiplicative noise. The  mutual dependence of $\phi$ and $\eta$ brings difficulty to study the explicit dependence of $\{\phi(t)\}_{t\geq 0}$ for the solution $\{\eta(t)\}_{t\geq 0}$ to \eqref{E: stable_eta} pathwisely. Nonetheless, we are able to take  advantage of the compactness of $[0, 2\pi]$ and start with investigating the bounds for 
the stable marginals based on Assumptions \ref{ass: A2-2} and \ref{ass: G-2}.  

\begin{lem}\label{lem: bound}
Let Assumptions \ref{ass: A1-2}, \ref{ass: A2-2} and \ref{ass: G-2} be satisfied. \ym{Suppose that $\eta(0)\in\hha$}, then there exists a $C>0$ such that for 
 sufficiently small $\eps>0$, 
$$\sup_{t\geq 0}\eee\|\eta(t)\|_\alpha^2\leq C.$$
\end{lem}
\begin{proof}
Consider Yosida approximation $\as_{s,n}:=n\as_s(nI-\as_s)^{-1}$ of $\as_s$.
 We denote $\eta_n$ by the solution to 
 $$ d\eta_n=-\varepsilon^{-2}\as_{s,n}\eta_n dt+ G_s(\phi,\eta_n)dW_t,\;\eta_{n}(0)=\eta(0). $$
Apply It\^{o}'s formula to $\|\eta_n(t)\|_\alpha^2$, then
\begin{equation*}
    \begin{split}
      d\|\eta_n(t)\|_\alpha^2
      =&-2\eps^{-2}\langle \as_{s,n}\eta_n(t),\eta_n(t)\rangle_\alpha dt+\|G_s(\phi(t),\eta_n(t))\|^2_{\LL_2}dt\\
      &+2\langle \eta_n(t),\; G_s(\phi(t),\eta_n(t))dW_t\rangle_\alpha.  
    \end{split}
\end{equation*}
Taking the expectation, using the property of $\as_{s,n}$ and $G_s$, 
for some $\eps$ within a small range,
there exists an $w>0$, $\tilde{w}>0$ and $\tilde{C}>0$ such that for all $t\geq 0$, 
\begin{equation*}
\begin{split}
      \frac{d\eee\|\eta_n(t)\|_\alpha^2}{dt}&=\eee\left\{-2\eps^{-2}\langle \as_{s,n}\eta_n(t),\eta_n(t)\rangle_\alpha +\|G_s(\phi(t),\eta_n(t))\|^2_{\LL_2}\right\}\\
   & \leq -2\eps^{-2}w \eee\|\eta_n(t)\|_\alpha^2+\ell_2\eee|\cos(\phi(t))+\sin(\phi(t))|^2+\ell_2\eee\|\eta_n(t)\|_\alpha^2\\
    & \leq -\tilde{w} \eee\|\eta_n(t)\|_\alpha^2+\tilde{C}
\end{split}
\end{equation*}
It follows from Gronwall’s inequality that, \ym{for each $n$, we have} $\eee\|\eta_n(t)\|_\alpha^2<C$ for every $t\geq 0$ and some $C>0$. The conclusion follows by sending $n$ to infinity. 
\end{proof}

\begin{lem}\label{lem: bound_sup_linear}
Let the assumptions in Lemma \ref{lem: bound} be satisfied. Fix  any  $T>0$ and any $p\geq 2$. 
For any initial condition $\eta(0)\in\hha$ a.s., there exists some $C>0$ such that
$$\eee\sup_{0\leq t\leq T}\|\eta(t)\|_\alpha^p\leq \|\eta(0)\|_\alpha^p+ C\eps^{p}.$$
\end{lem}
\begin{proof}
For  $t\in[0,T]$, the mild solution is given as 
\begin{equation}
    \eta(t)=e^{\eps^{-2}t\as_s}\eta(0)+\int_0^te^{\eps^{-2}(t-s)\as_s}G_s(\phi(s),\eta(s))dW_s.
\end{equation}
Let $P_sW_{\as}^G(t):=\int_0^te^{\eps^{-2}(t-s)\as_s}G_s(\phi(s),\eta(s))dW_s$ denote the stochastic convolution. 
The bound for the stochastic convolution follows \cite[Proposition 7.3]{da2014stochastic}. Indeed, by Lemma \ref{lem: bound}, there exists some $C>0$ and $C'>0$, such that
$$\int_0^T\eee\|G_s(\phi(s),\eta(s))\|^p_{\ls_2(\vs,\hh_\alpha)}ds\leq C\int_0^T\eee\|\cos(\phi(s))+\sin(\phi(s))+\eta(s)\|_\alpha^pds<C'<\infty. $$
Therefore $G_s(\phi(s),\eta(s))$ is $\LL^2$ predictable and there exists  constants $C_T>0$ and $C_T'>0$ such that
\begin{equation}
    \begin{split}
        &\eee\sup_{0\leq t\leq T}\left\|P_sW_{\as}^G(t)\right\|_\alpha^p
        \leq \eps^{p}C_T\eee\left(\int_0^T\|G_s(\phi(s),\eta(s)))\|_{\ls_2(\vs,\hh_\alpha)}^p ds\right)\leq  \eps^{p}C_T',
    \end{split}
\end{equation}
which implies that $\eee\sup_{0\leq t\leq T}\|\eta(t)\|_\alpha^p\leq \|\eta(0)\|_\alpha+C_T'\eps^p$.
\end{proof}

For test functions \ym{$f\in C_b^2([0, 2\pi]\times\mathcal{H}_{s})$}, the transition semigroup of \eqref{E: stable_phi} and \eqref{E: stable_eta} is such that $\trans_tf=\eee[f(\phi(t),\eta(t))|(\phi,\eta)]$.  Based on the compactness of $[0, 2\pi]$ and the above uniform bounds for $\{\eta(t)\}_{t\geq 0}$, the existence of invariant measure is guaranteed. 
\begin{prop}\label{prop: inv_phi}
Let the assumptions in Lemma \ref{lem: bound} be satisfied. Then
there exists an invariant measure for the transition semigroup $\{\trans_t\}_{t\geq 0}$ of \eqref{E: stable_phi} and \eqref{E: stable_eta} on $[0,2\pi]\times\hh_s$.
\end{prop}
\begin{proof}
\ym{We show the sketch of the proof. By the stability assumptions on $\as_s$, one can check that
$$\int_0^1 t^{-2\alpha}\|S(t)\|_{\LL_2(\hh,\hh)}^2dt<\infty$$
for each $\alpha\in(0,1]$, where $S(t):= e^{\eps^{-2}t\as_s}$. This implies that $\as_s$ generates a compact semigroup on $\hh$. It follows that, by introducing the compact operator $\mathcal{G}_{\alpha}: L^p([0,T]; \hh))\rightarrow C([0,T]; \hh)$ for $0<1/p< \alpha\leq 1$ and $t\in[0,T]$:
$$\mathcal{G}_{\alpha}f(t)=\int_0^t (t-s)^{\alpha-1}S(t-s)f(s)ds,\;\;f\in L^p([0,T], \hh),$$
as well as $Y_{\alpha}(t)=\int_0^t (t-r)^{-\alpha}S(t-r)G_s(\phi(r),\eta(r))dW(r),$
the mild solution can be expressed as
\begin{equation} \label{E:mildv}
y(t)=S(t)y_0+\frac{\sin\alpha\pi}{\pi}\mathcal{G}_{\alpha}(Y_{\alpha})(t).
\end{equation}
The compactness of $\mathcal{G}_{\alpha}$ has been shown in \cite[Proposition 8.4]{da2014stochastic} based on the compactness of $\{S(t)\}_{t\geq 0}$. 

Let $\mathscr{L}(\cdot)$ denote the law of random variables on the canonical space generated by $[0, 2\pi]\times\hh_s$.  By the compactness of $[0, 2\pi]$  as well as the boundedness given in Lemma \ref{lem: bound} and \ref{lem: bound_sup_linear},  applying \cite[Proposition 6]{da1992invariant}, it is straightforward to show that 
$$\left\{\frac{1}{t_n}\int_0^{t_n}\mathscr{L}(\phi(s),\eta(s))ds\right\}$$ forms a tight family of measure.   The existence of invariant measure for  $\{(\phi(t),\eta(t))\}_{t\geq 0}$ under $\trans_t$ follows by Krylov-Bogoliubov's Theorem  (along the same time sequence).
}
\end{proof}

\subsection{Transient Dissipativity of the Stable Modes}\label{sec: stable_trans}
The following lemma shows the approximated dissipativity condition given any transient transitions.
\begin{lem}\label{lem: dis}
Let the assumptions in Lemma \ref{lem: bound} be satisfied. For each arbitrarily small $\eps>0$ and fixed $\phi\in[0, 2\pi]$, there exists $w>0$ such that, for all $\eta_1,\eta_2\in\hha$, 
$$-2\eps^{-2}\langle \as_{s,n}(\eta_1-\eta_2),(\eta_1-\eta_2)\rangle_\alpha +\|G_s(\phi,\eta_1)-G_s(\phi,\eta_2)\|^2_{\LL_2}\leq -\eps^{-2}w\|\eta_1-\eta_2\|_\alpha,$$
where $\as_{s,n}$ is the  Yosida approximation of $\as_s$.
\end{lem}
\begin{proof}
The conclusion can be obtained under the assumptions considering sufficiently small $\eps>0$.
\end{proof}

\ym{\begin{rem}\label{rem: L2_limit}
For each fixed $\phi\in[0, 2\pi]$, we can denote $\{\eta^{\phi,\eta_0}(t)\}_{t\geq 0}$ as the solution to \eqref{E: stable_eta} with $\eta(0)=\eta_0$ a.s..
Note that by a similar approach as in \cite[Theorem 11.30]{da2014stochastic}, we can verify that for each fixed $\phi\in[0, 2\pi]$, given the initial condition $\eta_0=\zero$ a.s., the probability law of $\eta^{\phi,\zero}(t)$ converges weakly to the law of  a  random variable $\yk^\phi\in\LL^2(\Omega,\hha)$ as $\eps^{-2}t\rightarrow\infty$. 

Indeed, one can consider equation \eqref{E: stable_eta}
on the whole real line and denote $\{\eta^{\phi,\eta_0}_\tau(t)\}_{t\geq -\tau}$ by the solution that start at time $-\tau$ in $\eta_0$. Then $\mathscr{L}(\eta^{\phi,\eta_0}(\tau))=\mathscr{L}(\eta^{\phi,\eta_0}_\tau(0))$ for each $\tau\geq 0$.  
By similar proof as in Lemma \ref{lem: dis}, we have the mean square stability for each $\phi\in[0, 2\pi]$:
$$\eee\|\eta^{\phi,\zero}_\tau(0)-\eta^{\phi,\zero}_\sigma(0)\|_\alpha^2\leq MCe^{-\eps^{-2}w \tau}, \;\;\sigma>\tau.$$
The above inequality demonstrates that the Cauchy sequence in $\LL^2$ is dominated by $Ce^{-\eps^{-2}\omega \tau}$, and as $\eps^{-2}\tau\rightarrow\infty$, there exists a unique limit $\yk^\phi$  whose law is the required weak limit of $\{\mathscr{L}(\eta^{\phi,\zero}(t))\}_{t\geq 0}$. 
\end{rem}}

\ym{By virtue of Remark \ref{rem: L2_limit}},  the following proposition shows the transient behavior of the transition along the $\hh_s$ subspace. \ym{It is worth noting that the result only trivially describes the stable marginal of an invariant measure.}
\begin{prop}
As $\eps\ra 0$, at each each $t>0$, the marginal transition probability $H_{t}(\cdot\;|\;\phi,\eta)$ of \eqref{E: stable_eta}  
converges weakly a  measure $\nu^\phi(d\eta)$ on $\hh_s$ that only depends on $\phi$.
\end{prop}

\begin{proof}

By  Remark \ref{rem: L2_limit}, for $\eta(0)=0$ and each $\phi\in[0, 2\pi]$, there exists a unique limit $\yk^{\phi}$ with probability law $\nu^\phi(d\eta)$  as $\eps^{-2}t\rightarrow\infty$. Note that as $\eps\ra 0$, at each $t> 0$, we have $\eps^{-2}t\ra \infty$. Therefore, the marginal transition of $\eta^{\phi,\zero}(t)$ is given as  
\begin{equation}
    \begin{split}
        H_{t}(d\eta_t\;|\;\phi,\eta) &=H_{t}(d\eta_t\;|\;\phi,\eta)\mathds{1}_{\{\eta=\zero\}}\\
        &\approx \nu^\phi(d\eta_t).
    \end{split}
\end{equation}

We now consider a random initial distribution and let $\mathscr{L}(\eta(0))=\nu_0$. Since we have $\eee[\eta^2(0)]<\infty$, by a similar argument as Lemma \ref{lem: bound}, we can show that 
\begin{equation}\label{E: init}
    \lim_{\eps^{-2}t\rightarrow\infty}\eee\|\eta^{\phi,\zero}(t)-\eta^{\phi,\nu_0}(t)\|_\alpha^2=0,\;\;\phi\in[0, 2\pi],
\end{equation}
where $\eta^{\phi,\zero}(t) $ denotes the solution of \eqref{E: stable_eta} with $\eta(0)=0$ a.s. and $\eta^{\phi,\nu_0}(t) $ denotes the solution of \eqref{E: stable_eta} with $\mathscr{L}(\eta(0))=\nu_0$
for some fixed $\phi$. We aim to show that for any test function $f\in C_b( \hh_s)$, each of $\{\eta^\phi(t)\}_\phi$ with an arbitrary initial distribution converges weakly to the same limit point with probability law in $\{\nu^\phi \}_\phi$ for $\nu^\phi(d\eta_t):= \lim\limits_{\eps^{-2}t\rightarrow\infty}H_t(d\eta_t\;|\;\phi,\eta)\nu_0(d\eta)$, i.e.,  
\begin{equation}
    \int_{\hh_s} f( \eta(t))H_t(d\eta_t\;|\;\phi,\eta)\nu_0(d\eta)\xrightarrow{\eps^{-2}t\rightarrow \infty} \int_{\hh_s} f( \eta(t))\nu^\phi(d\eta_t),\;\;\phi\in[0, 2\pi].
\end{equation} 

Following the approach in \cite[Theorem 1]{da1992invariant}, we can show that for each fixed $\phi$,
\begin{equation}\label{E: converge_feller}
    \begin{split}
         &\left|\eee[f(\eta^{\phi,\nu_0}(t))]-\int_{ \mathcal{H}_s} f( \eta)\nu^\phi(d\eta)\right|\\
        \leq & \eee\left|f(\eta^{\phi,\nu_0}(t))-f(\eta^{\phi,0}(t))\right|+\left|\eee[f(\eta^{\phi,0}(t))]-\int_{ \mathcal{H}_s} f( \eta)\nu^\phi(d\eta)\right|\\
        =: & I_1(t)+I_2(t)
    \end{split}
\end{equation}
where $I_2(t)\rightarrow 0$ as discussed above, $I_1(t)$ is arbitrarily small. Indeed, 
\begin{equation}
\begin{split}
        I_1(t)&\leq C\cdot\ppp[\|\eta_0\|_\alpha\geq R] + \eee[\mathds{1}_{\{\|\eta_0\|_\alpha\leq R\}}\cdot f(\eta^{\phi,\nu_0}(t))-f(\eta^{\phi,0}(t))]\\
        &=: I_3(t)+I_4(t)
\end{split}
\end{equation}
where the constant $C$ in $I_3(t)$ is by the boundedness property of $f$. For arbitrary $\varsigma>0$, since $\eta_0\in \LL^2(\Omega;\hha)$, there exists $R>0$ such that 
$$\ppp[\|\eta_0\|_\alpha\geq R]<\varsigma. $$
We choose $R$ based on an arbitrary small $\varsigma$.
Note that $I_4(t)$ is restricted in a compact subspace and $f$ becomes uniformly continuous, by the property of $f$ and \eqref{E: init}, $I_4(t)\rightarrow 0$ as $\eps^{-2}t\rightarrow \infty$. 

We have seen that   $\nu^\phi(d\eta)$ is unique w.r.t. each $\phi$ with arbitrary initial distribution $\nu_0$, including $\delta_\eta$. The statement hence follows. 
\end{proof}

\begin{rem}\label{rem: disintegration}
Unlike the case where the linearized equations have no coupling effects, we are not able to explicitly solve the invariant measure for  $\{\eta(t)\}_{t\geq 0}$ by considering the transitions separately. 

The above proposition only provides a view that the marginal transition along $\hh_s$ quickly forgets the initial point $\eta$ for sufficiently small noise. In this view, we can represent an invariant measure by a disintegrated form
\begin{equation*}
    \begin{split}
       \mu(d\phi\times d\eta) &=\int_{[0, 2\pi]\times\hh_s}R(d\phi\;|\;\phi,\eta)\nu^\phi(d\eta)\mu(d\phi\times d\eta) \\
       &=:\tilde{\mu}(d\phi)\tilde{\nu}^\phi(d\eta).
    \end{split}
\end{equation*}
Note that   $\tilde{\mu}(d\phi)$ and $\tilde{\nu}^\phi(d\eta)$ are difficult to solve explicitly. This, however, motivates us to deliver an approximation in Section \ref{sec: lya_approx} of the invariant measure with the above disintegrated form.
\end{rem}

\subsection{Conditions on Uniqueness of Invariant Measure}
We have seen in Section \ref{sec: stable_trans} that the disintegration measure $\nu^\phi(d\eta)$ along $\hh_s$ uniquely exists given any $\phi$. Similarly, under condition that $G_c^\phi(G_c^\phi)^*\neq 0$ for all $\phi$ \cite{khas1967necessary}, for each $\eta$, the solution of  \begin{equation}
    d\phi=(\varepsilon^{-2}b_c^\cf+b_c^\qf)dt+\Gamma (\phi,\eta) dt-G_c^\phi(\phi,\eta)dW_t
\end{equation} 
admits a unique limit measure that is solved by  the associated Fokker-Plank equation  
\begin{equation}\label{E: fk}
    \frac{dp}{d\phi}\big(\varepsilon^{-2}b_c^\cf+b_c^\qf+\Gamma(\phi,\eta)\big)+\frac{1}{2}\frac{d^2p}{d\phi^2}\left[G_c^\phi(G_c^\phi)^*\right](\phi,\eta)=0,
\end{equation}
where $p$ is the density function.

However, as discussed in Remark \ref{rem: disintegration}, it will
not be enough to consider ergodicity or uniqueness of invariant measure for $\{\phi(t)\}_{t\geq 0}$ and $\{\eta(t)\}_{t\geq 0}$ separately. It is not sufficient to only suppose the full-rank property of $G_c^\phi$. We hence impose a set of extra stronger conditions to guarantee the uniqueness of the invariant measure. 
\begin{ass}\label{ass: full-rank}
For any $\alpha\in(0,1]$, we assume that the operator $G(u)$ is   invertible for each $u\in\hh_\alpha\setminus\{\zero\}$.
\end{ass}

The above assumption plays a role as the Lie algebra condition to guarantee the uniqueness of the $\{\trans_t\}_{t\geq 0}$ for the coupled linearized system. It is equivalent to verify the non-singular condition, i.e., we need $G(u)$ to be bounded from below in the following sense:
\begin{equation}
    \langle G(u)v,v\rangle_\alpha\geq m\|u\|_\alpha \|v\|, \;m>0, v\in\mathcal{V}.
\end{equation}
Combining with Assumption \ref{ass: G-2} and \ref{ass: A2-2}, it can be verified that $G(u)$ is holomorphic on $\hh_\alpha\setminus\{\zero\}$.  

\section{Asymptotic Approximation of the Top Lyapunov Exponent}\label{sec: lya_approx}

Motivated by Remark \ref{rem: disintegration},  we derive the asymptotic expansion of the invariant measure $\mu^\eps$ in this section for the approximation of the top Lyapunov exponent. By \cite[Theorem 9.25]{da2014stochastic}, we can show that for any test function $f\in C_b^2([0, 2\pi]\times\mathcal{H}_{s})$, the quantity $\trans_tf$ satisfies
\begin{equation*}
\begin{split}
      \lim\limits_{t\downarrow 0}\frac{\trans_tf(\phi,\eta)-f(\phi,\eta)}{t} =\lle f(\phi,\eta)
\end{split}
\end{equation*}
where
\begin{equation}\label{E: generator_polar}
    \lle=\frac{1}{\eps^2}\LLk_0+\LLk_1^\qf,
\end{equation}
and 
$$\LLk_0(\cdot)=\left[b_c^\cf\frac{\partial }{\partial \phi}+\as_s\eta \frac{\partial }{\partial \eta}\right](\cdot)$$
$$\LLk_1^\qf(\cdot)=\left[(b_c^\qf+\Gamma(\phi,\eta)\frac{\partial }{\partial \phi}\right](\cdot)+\frac{1}{2}\operatorname{tr}\left[
\frac{\partial^2(\cdot)}{\partial\eta^2}
G_sG_s^*+\frac{\partial^2(\cdot)}{\partial\phi^2}G_c^\phi (G_c^\phi)^*\right](\phi,\eta)$$

\begin{rem}
To simplify the notation, we use $\LLk^\eps$ and $\LLk_1$ in stead of $\lle$ and $\LLk_1^\qf$ in this section. 
\end{rem}

For any test function $f\in C^2([0, 2\pi]\times\hh_s)$ one should have 
$\langle\LLk^\eps f,\mu^\eps\rangle=0.$
We expand $\mu^\eps$ as \begin{equation}
    \mu^\eps=\mu_0+\eps^2\mu_1+\mathcal{O}(\eps^3), 
\end{equation}
then, respectively on the level $\mathcal{O}(\eps^{-2})$ and $\mathcal{O}(1)$, 
\begin{subequations}\label{E: expansion}
\begin{align}
& \langle\LLk_0f, \mu_0\rangle=0 \label{E: level0}\\
& \langle\LLk_0 f,\mu_1\rangle=-\langle \LLk_1f,\mu_0\rangle \label{E: level1}
\end{align}
\end{subequations}

We proceed to find the solutions to \eqref{E: expansion}. 
Here we adopt a method similar to \cite{sri2012almost} to evaluate the first order asymptotic expansion of the top Lyapunov exponent. The proofs related to solving \eqref{E: expansion} are completed in the Appendix.
\begin{prop}\label{prop: 4.1}
$\mu_0(d\phi\times d\eta)=\frac{d\phi}{2\pi}\delta_0(d\eta)$ is an ergodic measure for Eq.(\ref{E: level0}).
\end{prop}

 To solve \eqref{E: level1}, we first calculate R.H.S. of \eqref{E: level1} using the following lemma. 
\begin{lem}\label{lem: rhs}
\begin{equation}
    \begin{split}
     -\langle \LLk_1f,\mu_0\rangle=&\int_0^{2\pi}\frac{\partial \Gamma}{\partial\phi}(\phi,0)f(\phi,0)\frac{d\phi}{2\pi}-\int_0^{2\pi}\left\{f(\phi,\eta)\tr[(D_\phi G_c^\phi)(D_\phi G_c^\phi)^*]\right\}_{\eta=0}\frac{d\phi}{2\pi}\\
     &-\frac{1}{2}\int_0^{2\pi}\left\{\sum_{k,j\in\mathbb{Z}_s} [G_s G_s^*]_{kj} f_\eta''(\phi,\eta; e_k,e_j)\right\}_{\eta=0} \frac{d\phi}{2\pi},  
    \end{split}
\end{equation}
where $[G_s G_s^*]_{kj}:=\langle G_s G_s^*e_k,e_j\rangle$,  
and $f_\eta''(\phi,\eta; e_k,e_j)$  is the Fr\'{e}chet derivative w.r.t. $\eta$ along $e_k$ and $e_j$.
\end{lem}

Observing the above, we try ansatz of the following form to match the R.H.S. of \eqref{E: level1}:
\begin{equation}\label{E: ansatz}
    \mu_1(d\phi\times d\eta)=\frac{d\phi}{2\pi}\kappa(\phi)\delta_0(d\eta)+\frac{d\phi}{2\pi}\frac{\partial^2\delta_0}{\partial\eta^2}(\chi(\phi), h)(d\eta),
\end{equation}
where $\kappa: [0, 2\pi]\rightarrow\R$, $\chi: [0, 2\pi]\rightarrow\hh_s$, and an arbitrary $h=\sum_{k\in\mathbb{Z}_s} \langle h,e_k\rangle e_k\neq 0$ (recall notation $\mathbb{Z}_s$ in Definition \ref{def: other_notation-2}) that can make the calculation simple. 

\begin{rem}
The expression $\frac{\partial^2\delta_0}{\partial\eta^2}(a, b)(d\eta) $ appears in the above ansatz is the measure on $\mathcal{H}_s$ in the sense of directional distribution, where $a$ and $b$ are the directions.
For Fr\'{e}chet differentiable test function $f(\phi,\eta)$, we define the  Fr\'{e}chet derivatives 
$$f'(\phi,\eta; a):=\frac{\partial f}{\partial \eta} (\phi,\eta)(a)\;\;\text{and}\;\;f''(\phi,\eta; a,b):=\frac{\partial^2f}{\partial\eta^2}(\phi,\eta)(a,b).$$
Then the distributional measure should satisfy
$$\left\langle f(\phi, \eta), \frac{\partial\delta_0}{\partial \eta}(a)(d\eta) \right\rangle= \left\langle \frac{\partial f}{\partial \eta}(\phi,\eta)(a),  \delta_0(d\eta) \right\rangle=-f'(\phi,0;a),$$
and, likewise,
$$\left\langle f(\phi, \eta), \frac{\partial^2\delta_0}{\partial \eta^2}(a,b)(d\eta) \right\rangle= f''(\phi,0;a,b).$$
\end{rem}

\begin{prop}\label{prop: 4.3}
Let $\kappa(\phi)$ and $\chi(\phi)$ be the notions in the  ansatz \eqref{E: ansatz}. Let $\chi_k(\phi):=\langle \chi(\phi),e_k\rangle$ for $k\in\mathbb{Z}_s$ and choose $ h=\sum_{k\in\mathbb{Z}_s}h_ke_k\in\hha$, where $h_k=\frac{1}{2^{|k|+2}(1-\rho_k)}$ and $\rho_k$ is the $k$-th eigenvalue defined in Assumption \ref{ass: A1-2}. Suppose $\kappa(\phi)$
solves 
\begin{equation}\label{E: kappa}
    -b_c^\cf\frac{\partial\kappa}{\partial\phi}(\phi)=\frac{\partial\Gamma}{\partial\phi}(\phi,0)-\tr[(D_\phi G_c^\phi)(D_\phi G_c^\phi)^*](\phi,0).
\end{equation}
and $\chi_k(\phi)$ solves 

\begin{equation}\label{E: zeta}
    \left(b_c^\cf\frac{\partial}{\partial\phi}-\rho_k+1\right)\chi_k(\phi)= -\sum_{j\in\mathbb{Z}_s}\left\{[G_s G_s^*]_{kj}\right\}_{\eta=0}
\end{equation}
for all $k\in\mathbb{Z}_s$.
Then,
$$\mu_1(d\phi,d\eta)=\frac{d\phi}{2\pi}\kappa(\phi)\delta_0(d\eta)+\frac{d\phi}{2\pi}\frac{\partial^2\delta_0}{\partial\eta^2}(\chi(\phi), h)(d\eta).$$
\end{prop}

By solving \eqref{E: kappa} and \eqref{E: zeta}, we are able to obtain the exact form of $\mu_1$ as in \eqref{E: ansatz}.

Given the assumptions on $G$, the terms $\frac{\partial\Gamma }{\partial \phi}$ and $ \;\tr[( G_c^\phi)(G_c^\phi)^*],\;\tr[(D_\phi G_c^\phi)(D_\phi(G_c^\phi)^*]$ are Lipschitz continuous in $\phi$, the existence of solutions is  guaranteed. 
Based on the differentiability assumption of $G$, the solution to \eqref{E: kappa} is 
\begin{equation}
    \kappa(\phi)=-\frac{1}{b_c^\cf}\left(\Gamma(\phi, 0)+\frac{1}{2}\tr[G_c^\phi (D_\phi G_c^\phi)^*+(D_\phi G_c^\phi) (G_c^\phi)^*](\phi,0)\right) + 1, 
\end{equation}
where the constant $1$ comes from the marginal integral over $\phi$. 
Note that \eqref{E: zeta} can also be represented as 
$$\left(b_c^\cf\frac{\partial}{\partial\phi}-\as_s+1\right)\chi(\phi)=-\sum_{k,j\in\mathbb{Z}_s}\left\{[G_s G_s^*]_{kj}\right\}_{\eta=0}<\infty, $$
since the operator on the L.H.S. is both unbounded in $\mathcal{S}$ and $\hh_s$, we have 
$$\chi(\phi)=-\left(b_c^\cf\frac{\partial}{\partial\phi}-\as_s+1\right)^{-1}\sum_{k,j\in\mathbb{Z}_s}\left\{[G_s G_s^*]_{kj}\right\}_{\eta=0}$$ well defined. 
For short, we let $G_k^R+iG_k^I=-\sum_{j\in\mathbb{Z}_s}\left\{[G_s G_s^*]_{kj}\right\}_{\eta=0}$, then
for each $k\in\mathbb{Z}_s$, 
$$b_c^\cf\frac{\partial
}{\partial \phi}\begin{bmatrix}
\chi_k^R\\ \chi_k^I
\end{bmatrix}
(\phi)-\begin{bmatrix}
a_k-1 & -b_k\\ b_k & a_k-1
\end{bmatrix}\chi_k(\phi)=\begin{bmatrix}
G_k^R\\ G_k^I
\end{bmatrix}, $$
and

\begin{equation}\label{E: chik}
    \begin{bmatrix}
\chi_k^R\\ \chi_k^I
\end{bmatrix}(\phi)=\begin{bmatrix}
a_k-1 & -b_k\\ b_k & a_k-1
\end{bmatrix}^{-1}\left(e^{\vartheta_k}\begin{bmatrix}
\cos(w_k) & -\sin(w_k)\\ \sin(w_k) & \cos(w_k)
\end{bmatrix}-\begin{bmatrix}
G_k^R\\ G_k^I
\end{bmatrix}\right)
\end{equation}
is the solution, where $\vartheta_k:=\frac{(a_k-1)\phi}{b_c^\cf}$ and $w_k:=\frac{(b_k-1)\phi}{b_c^\cf}$. Then 
\begin{equation}\label{E: chi_phi}
    \chi(\phi)=\sum_{k\in\mathbb{Z}_s}(\chi_k^R(\phi)+i\chi_k^I(\phi))e_k.
\end{equation}

Now that 
$ \lambda^\eps=\langle\Qs,\mu^\eps\rangle=\langle\Qs,\mu_0\rangle+\eps^2\langle\Qs,\mu_1\rangle+r(\eps),$
where $r(\eps)$ represents the remainder. By a similar argument as \cite[Section 3]{arnold1986asymptotic} and \cite[Lemma 4.3]{sri2012almost}, we show that $r(\eps)=\mathcal{O}(\eps^3)$.

\begin{prop}
For the generator $\LLk_\eps=\frac{1}{\eps^2}\LLk_0+\LLk_1$, there exists functions $F_0, F_1$ 
on $[0, 2\pi]\times\hh_s$ and functions $\tilde{f}_0,\tilde{f}_1$ 
that are independent of $ \mathcal{S}\times\hh_s$, such that the sequence of Poisson equations 
\begin{equation}
    \begin{split}
        \LLk_0F_0\quad\quad\;\;\quad&=\zeta-\tilde{f}_0\\
        \LLk_0F_1+\LLk_1F_0 &=\;\;\;\;-\tilde{f}_1\\
    \end{split}
\end{equation}
are satisfied. As a consequence, 
\begin{equation}
    \begin{split}
        r(\eps)=-\eps^{3}[\langle\LLk_1F_1,\mu^\eps\rangle+\langle \LLk_1(F_0+\eps^2 F_1),\mu_1\rangle-\langle \LLk_1F_1,\mu_0+\eps^2\mu_1\rangle].
    \end{split}
\end{equation}
\end{prop}
Given the boundedness  $\sup_{\phi,\eta}\{|\LL_1F_1|,|\LL_1F_0|\}= C$, we immediately have
$|r(\eps)|<\os(\eps^3)$. 
Eventually, combining with $\mu^\eps$, we have \begin{equation}
\begin{split}
        \lambda^\eps=&\langle\Qs,\mu_0\rangle+\eps^2\langle\Qs,\mu_1\rangle+\mathcal{O}(\eps^3)\\
        =&\frac{1}{2\pi}\int_0^{2\pi}\Qs(\phi,0)d\phi+\frac{\eps^2}{2\pi}\int_0^{2\pi}\Qs(\phi,0)\kappa(\phi)d\phi\\
        &+\frac{\eps^2}{2\pi}\int_0^{2\pi} \Qs''\left(\phi,0;\sum_{k\in\mathbb{Z}_s}(\chi_k^R(\phi)+i\chi_k^I(\phi))e_k,\sum_{k\in\mathbb{Z}_s}\frac{1}{2^{|k|+2}(1-\rho_k)}e_k\right)d\phi\\
       &+\mathcal{O}(\eps^3).
\end{split}
\end{equation}

\section{Conclusion}\label{sec: conclusions}
 In this paper, we
provide for the first time a derivation of an asymptotic expansion for the top
Lyapunov exponent for {SPDE}s with multiplicative noise when the parameter moves slowly through the deteriministic Hopf bifurcation point. Instead of obtaining a dimension reduction using homogenization, the formula of top Lyapunov exponent was provided explicitly. 
We prove the existence of invariant measure on the product space of the unit sphere  and the stable mode, and show the conditions for ergodicity.  The disintegrated form of invariant measure as in Remark \ref{rem: disintegration}  explains the long term dependence of the stable marginals on the unit sphere of the critical mode.  However, since it is  difficult to solve, 
we derive an asymptotic expansion
of the invariant measure of the disintegrate form and apply it in the Furstenberg–Khasminskii formula for the top
Lyapunov exponent.  

\section*{Acknowledge}
The authors acknowledge partial support for this work from National Sciences and Engineering Research Council~(NSERC) Discovery grant 50503-10802.

\appendix
\section*{Appendix: Proofs in Section \ref{sec: lya_approx}}
\textbf{Proof of Proposition \ref{prop: 4.1}:}
Note that $\LLk_0$ behaves like deterministic:   $\eta(t)\rightarrow 0$ due to the stable semigroup generated by $\as_s$. Rigorously, 
\begin{equation*}
    \begin{split}
        \langle\LLk_0 f, \mu_0\rangle & =\iint_{[0, 2\pi]\times \mathcal{H}_s}b_c^\cf\frac{\partial f}{\partial \phi}\frac{d\phi}{2\pi}\delta_0(d\eta)+\iint_{[0, 2\pi]\times \mathcal{H}_s}\as_s \eta\frac{\partial f}{\partial\eta}(\phi,\eta)\frac{d\phi}{2\pi}\delta_0(d\eta)\\
        &=\int_0 ^{2\pi}b_c^\cf\frac{df(\phi,0)}{2\pi}-0=0. 
    \end{split}
\end{equation*}
\qed

\noindent\textbf{Proof of Lemma \ref{lem: rhs}:} 
\begin{equation*}
    \begin{split}
        -\langle \LLk_1f,\mu_0\rangle =& -\left\langle[ b_c^\qf+\Gamma(\phi,\eta)]\frac{\partial}{\partial \phi}f,\; \frac{d\phi}{2\pi}\delta_0(d\eta)\right\rangle-\frac{1}{2}\left\langle\frac{\partial^2f}{\partial\phi^2}\tr[G_c^\phi (G_c^\phi)^*], \;\frac{d\phi}{2\pi}\delta_0(d\eta)\right\rangle\\
        &-\frac{1}{2}\left\langle\tr\left[\frac{\partial^2f}{\partial\eta^2}G_s G_s^*\right], \;\frac{d\phi}{2\pi}\delta_0(d\eta)\right\rangle\\
        =& \int_0^{2\pi}\frac{\partial \Gamma}{\partial\phi}(\phi,0)f(\phi,0)\frac{d\phi}{2\pi}-\int_0^{2\pi}\left\{f(\phi,\eta)\tr[(D_\phi G_c^\phi)(D_\phi G_c^\phi)^*]\right\}_{\eta=0}\frac{d\phi}{2\pi}\\
    &-\frac{1}{2}\int_0^{2\pi}\left\{\sum_{k,j\in\mathbb{Z}_s} [G_s G_s^*]_{kj} \left\langle \frac{\partial^2f}{\partial\eta^2} e_k,  e_j\right\rangle\right\}_{\eta=0} \frac{d\phi}{2\pi}\\
    =& \int_0^{2\pi}\frac{\partial \Gamma}{\partial\phi}(\phi,0)f(\phi,0)\frac{d\phi}{2\pi}-\int_0^{2\pi}\left\{f(\phi,\eta)\tr[(D_\phi G_c^\phi)(D_\phi G_c^\phi)^*]\right\}_{\eta=0}\frac{d\phi}{2\pi}\\
    &-\frac{1}{2}\int_0^{2\pi}\left\{\sum_{k,j\in\mathbb{Z}_s} [G_s G_s^*]_{kj} f_\eta''(\phi,\eta; e_k,e_j)\right\}_{\eta=0} \frac{d\phi}{2\pi}.
    \end{split}
\end{equation*}
\qed

\noindent\textbf{Proof of Proposition \ref{prop: 4.3}:}
For test function $f\in C^{1,2}([0, 2\pi]\times\hh_s)$, we have
\begin{equation}
    \begin{split}
        \left\langle\LLk_0 f,\; \frac{d\phi}{2\pi}\kappa(\phi)\delta_0(d\eta)\right\rangle = &\left\langle\left[b_c^\cf\frac{\partial }{\partial \phi}+\as_s\eta \frac{\partial }{\partial \eta}\right](f),\; \frac{d\phi}{2\pi}\kappa(\phi)\delta_0(d\eta)\right\rangle\\
        =&-\int_0^{2\pi}\left\{b_c^\cf\frac{\partial \kappa}{\partial\phi}(\phi)f(\phi,\eta)\right\}_{\eta=0}\frac{d\phi}{2\pi}\\
        =& \int_0^{2\pi}\frac{\partial \Gamma}{\partial\phi}(\phi,0)f(\phi,0)\frac{d\phi}{2\pi}\\
        &-\int_0^{2\pi}f(\phi,0)\tr[(D_\phi G_c^\phi)(D_\phi G_c^\phi)^*](\phi,0)\frac{d\phi}{2\pi}.
    \end{split}
\end{equation}
For each $k\in\mathbb{Z}_s$, 
\begin{equation}\label{E: derive1}
    \begin{split}
        & \left\langle\LLk_0 f,\; \frac{d\phi}{2\pi}\frac{\partial^2\delta_0}{\partial\eta^2}(\chi_k(\phi), h)(d\eta)\right\rangle = \left\langle b_c^\cf\frac{\partial f}{\partial \phi}+\as_s\eta \frac{\partial f}{\partial \eta},\;\frac{d\phi}{2\pi}\frac{\partial^2\delta_0}{\partial\eta^2}(\chi_k(\phi), h)(d\eta)\right\rangle\\
        = & -\left\langle\frac{\partial\delta_0}{\partial \eta}(h)(d\eta),\;-f_\eta'\left(\phi,\eta; b_c^\cf\frac{\partial\chi_k}{\partial\phi}(\phi)\right) \frac{d\phi}{2\pi}\right\rangle\\
        &-\left\langle\frac{\partial\delta_0}{\partial \eta}(h)(d\eta),\;[f_\eta''(\phi,\eta; \as_s\eta, \chi_k(\phi))+f_\eta'(\phi,\eta; \as_s
    \chi_k(\phi))]\frac{d\phi}{2\pi} \right\rangle\\
        = & \left\langle\delta_0(d\eta),\;-f_\eta''\left(\phi,\eta;\; b_c^\cf\frac{\partial\chi_k}{\partial\phi}(\phi),h)\right) \frac{d\phi}{2\pi}\right\rangle +
        \left\langle\delta_0(d\eta),\;f_\eta'''(\phi,\eta; \as_s\eta, \chi_k(\phi),h)\frac{d\phi}{2\pi} \right\rangle\\
        &+ \left\langle\delta_0(d\eta),\;[f_\eta''(\phi,\eta;\as_s h,\chi_k(\phi))+f_\eta''(\phi,\eta; \as_s\chi_k(\phi),h)]\frac{d\phi}{2\pi} \right\rangle\\
        = & -\int_0^{2\pi}\left\{f_\eta''\left(\phi,\eta;\; (b_c^\cf\frac{\partial}{\partial\phi}-\as_s)\chi_k(\phi),h\right)-f_\eta''(\phi,\eta;\chi_k(\phi),\as_s h)\right\}_{\eta=0} \frac{d\phi}{2\pi}.
    \end{split}
\end{equation}
By the hypothesis on $\chi_k(\phi)$ and $h$, the last line of the above can be expanded as
\begin{equation*}
    \begin{split}
       &-\sum_{j\in\mathbb{Z}_s}(1-\rho_j)h_j\int_0^{2\pi}\left\{f_\eta''\left(\phi,\eta;\; (b_c^\cf\frac{\partial}{\partial\phi}-\as_s)\chi_k(\phi),e_j)\right)+f_\eta''(\phi,\eta;\chi_k(\phi),e_j)\right\}_{\eta=0} \frac{d\phi}{2\pi}\\
       =& -\frac{1}{2}\int_0^{2\pi}\sum_{j\in\mathbb{Z}_s}\left\{[G_s G_s^*]_{kj}\right\}_{\eta=0}f''_\eta(\phi,0;e_k,e_j)\frac{d\phi}{2\pi}
    \end{split}
\end{equation*}
Combining this and \eqref{E: derive1}, we have
$$\left\langle\LLk_0 f,\; \frac{d\phi}{2\pi}\frac{\partial^2\delta_0}{\partial\eta^2}(\chi(\phi), h)(d\eta)\right\rangle =-\frac{1}{2}\int_0^{2\pi}\sum_{k,j\in\mathbb{Z}_s}\left\{[G_s G_s^*]_{kj}\right\}_{\eta=0}f''_\eta(\phi,0;e_k,e_j)\frac{d\phi}{2\pi}.$$
Thus, by Lemma \ref{lem: rhs}, we have 
$$ \left\langle\LLk_0 f,\; \mu_1\right\rangle= -\left\langle\LLk_1 f,\; \mu_0\right\rangle,$$
which completes the proof.\qed

\bibliographystyle{tfnlm}
\bibliography{ref.bib}

\end{document}